\documentclass{amsart}
\usepackage{amscd, amssymb, mathrsfs, mathabx, 
dsfont, xcolor, amsmath, soul,diagrams}

\usepackage{hyperref}  
\usepackage{collectbox}
\usepackage{enumitem}

\usepackage{tikz}
\usetikzlibrary{shapes, arrows}



\input{xy}
\xyoption{all}

\newtheorem*{theorem*}{Main Theorem}
\newtheorem*{lemma*}{Lemma A}

\newtheorem{theorem}{Theorem}[section]

\newtheorem{lemma}[theorem]{Lemma}

\theoremstyle{definition}
\newtheorem{definition}{Definition}

\theoremstyle{remark}
\newtheorem{remark}[theorem]{Remark}

\numberwithin{equation}{section}

\newcommand{\lan}{\langle}
\newcommand{\ran}{\rangle}

\newcommand{\sH}{\mathscr{H}}

\newcommand{\sR}{\mathscr{R}}

\newcommand{\Z}{\mathbb{Z}}

\newcommand{\ch}{\mathrm{ch}}
\newcommand{\Ch}{\mathrm{Ch}}

\newcommand{\fm}{\mathfrak{m}}
\newcommand{\MF}{\mathrm{MF}}
\newcommand{\id}{\operatorname{id}}

\newcommand{\cA}{{\mathcal{A}}}

\def\on{\operatorname}

\newcommand{\sK} {\mathscr{K}}

\def\Gres{\on{Res}^G}
\def\GGres{\on{Res}_{f}}

\def\res{\on{res}}

\newcommand{\xra}[1]{\xrightarrow{#1}}
\newcommand{\xla}[1]{\xleftarrow{#1}}

\newcommand{\ot }{\otimes}

\newcommand{\ccan}{{\mathsf{can}}}

\newcommand{\C}{{\mathbb{C}}}

\newcommand{\rH}{{\mathrm{H}}}
\newcommand{\cH}{{\mathsf{H}}}

\newcommand{\uu}{{(\!(u)\!)}}
\newcommand{\uuu}{{[\![u]\!]}}

\def\bu{\bullet}

\newcommand{\cC}{{\mathcal{C}}}

\def\on{\operatorname}

\newcommand{\Hom }{{\mathrm{Hom}}}

\newcommand{\op}{{\mathrm{op}}}

\newcommand{\Perf}{\mathrm{Perf}}

\newcommand{\Kw}{\sK_{f}}

\newcommand{\Cw}{\sH_{f}}
\newcommand{\Cwn}{\sH_{F}}
\newcommand{\CCw}{\sR_{f}}
\newcommand{\CCwn}{\sR_{F}}

\newcommand{\bbco}{{\nabla}_{{\!\!{\partial_u}}}^{f}}
\newcommand{\bbcoo}{{\nabla}_{{\!\!{\partial_{z_i}}}}^{F}}
\newcommand{\bbcon}{{\nabla}_{{\!\!{\partial_u}}}^{F}}

\newcommand{\bde}{{\boldsymbol \delta}}

 \newcommand{\bfx}{\mathbf{x}}
\newcommand{\bfy}{\mathbf{y}}

\newcommand{\HH}{{\mathbb H}}

\makeatletter
\@namedef{subjclassname@2020}{%
  \textup{2020} Mathematics Subject Classification}
\makeatother

\begin{document}

\title[Higher residues and canonical pairing]
      {Higher residues and canonical pairing on the twisted de Rham cohomology}  

\author[H. Kim]{Hoil Kim}
\address{Department of Mathematics\\
  Kyungpook National University\\
  Taegu 702-701, Republic of Korea}
\email{hikim@knu.ac.kr}

\author[T. Kim]{Taejung Kim }
\address{Department of Mathematics Education\\
  Korea National University of Education\\
250 Taeseongtabyeon-ro, Gangnae-myeon\\
Heungdeok-gu, Cheongju-si, Chungbuk 28173\\
 Republic of Korea}
\email{tjkim@kias.re.kr}

\thanks{H. Kim was supported by NRF-2018R1D1A1B07044575 and T. Kim was supported by NRF-2018R1D1A3B07043346.}

\begin{abstract}

        We describe an explicit formula of the canonical pairing on the twisted de Rham cohomology associated with the category of local matrix factorizations and by characterizing its relation to Saito's higher residue pairings we reprove the conjecture of Shklyarov.
\end{abstract}


\maketitle

 \section{Introduction}
 Higher residue pairings introduced by K. Saito have played an important role in the development of the theory of Frobenius manifold \cite{her3, tsaito,tu1, tu2} and the concept of higher residue pairings starts to appear in a computation of intersection numbers of cohomology classes through the Feynman integrals with the graphs and super potentials in the recent trends of the scattering theory; see \cite{mp, van}. While their importance has kept growing in mathematics and physics, information about a definite formula describing them does not seem to be readily available in the literature. In this paper, we obtain their explicit formula and apply it to solve the problems related to the canonical pairing in the dg category of matrix factorizations. More specifically, we describe the composition of following maps explicitly:
\begin{equation}\label{3.2}
  \begin{aligned}
    \cH_{f}^{(0)}&\times \cH_{f}^{(0)}\xra{\id\times(-1)^n}\cH_{f}^{(0)}\times \cH_{-f}^{(0)}\\
    &\xra{\cong}\rH_n(\Omega^{\bullet}_{Q_\fm/\C}\uuu, -d{f} +ud)\times\rH_n(\Omega^{\bullet}_{Q_\fm/\C}\uuu, d{f} +ud)\\
  &\xla{\cong} \mathbb{H}_n R\Gamma_{\fm} (\Omega^\bu_{Q_\fm/\C}\uuu, -df+ud) \times \mathbb{H}_n R\Gamma_\fm(\Omega^{\bullet}_{Q_\fm/\C}\uuu, d{f} +ud)\\
  & \xra{\wedge\circ \mathsf{kun}}
  \mathbb{H}_{2n} \big(\cC(x_1,\dots,x_n) \otimes_{Q_\fm}  (\Omega_{Q_\fm/\C}^\bu\uuu, ud) \big) \\
  & \xra{\on{\res}} \C\uuu
  \end{aligned}
  \end{equation}
where $\mathsf{kun}$ is the K\"unneth map; for the descriptions of the rest of notations in the above, see Section~\ref{localcoho}. The explicit formula of the composition of maps in \eqref{3.2} is given in \eqref{hres1}. From this, in Theorem~\ref{cojecmain} we establish that \eqref{hres1} multiplied by $u^n$ is in fact the explicit formula of Saito's higher residue pairings. In \cite{shkl}, D. Shklyarov gives  a categorical interpretation of Saito's higher residue pairings with the ambiguity of sign. The determination of the sign is known as the conjecture of Shklyarov. We discuss about the relation between \eqref{hres1} and Shklyarov's conjecture in Section~\ref{sec422}.

\subsection{Outline of paper}

In Section~\ref{localcoho} after briefly explaining relevant facts about the local cohomology and Grothendieck's residue map, we describe the composition of maps in \eqref{3.2}. In Section~\ref{32hr}, we calculate its explicit formula \eqref{hres1} with \eqref{important2}. In Section~\ref{decan} we introduce the characteristic equation~\eqref{eqn: char pairing} of the canonical pairing, which one can take as a definition of the canonical pairing without looking over its typical definition found in \cite{Shk: HRR, shkl}. From the properties in Section~\ref{gropro}, the explicit formula~\eqref{hres1} and Equation~\eqref{eqn: char pairing} enable us to deduce that \eqref{hres1} is the canonical pairing in Section~\ref{33eq}. Section~\ref{sec4} deals with the equivalence between \eqref{hres1} and Saito's higher residue pairings; see Theorem~\ref{cojecmain}.  As a byproduct, we give another proof of the conjecture of Shklyarov in Theorem~\ref{saitomain} in Section~\ref{sec422}.

\section{Local cohomology and higher residue  pairings $\sH_f$}

  \subsection{Local cohomology}\label{localcoho}
  Let $Q=\mathbb{C}[x_1,\dots,x_n]$ and $\fm=(x_1,\dots,x_n)$. The $k$th local cohomology  associated with the maximal ideal $\fm$ of a complex
  $(\Omega_{Q/\C}^\bu\uuu, -df+ud)$ where $u$ is a formal variable of degree 2 is by definition
  the $k$th hypercohomology $\mathbb{H}_k R\Gamma_{\fm} (\Omega^\bu_{Q/\C}\uuu, -df+ud)$. That is, after taking an injective resolution $\mathcal{I}^{\bullet,\bullet}$ of a complex $(\Omega_{Q/\C}^\bu\uuu, -df+ud)$,
  one takes the $k$th homology of the total complex $\bigoplus_n\oplus_{p+q=n} \mathcal{I}^{p,q}$. Note that for a local cohomology of a $Q$-module $M$, we simply write  $\rH^{k}_{\fm}(M)$. On the other hand, it can be also calculated using the \v{C}ech complex as follows. Let
  $$\cC(x_1,\dots,x_n)=\bigoplus_{j=0}^{n}\cC^j\text{ with }\cC^j=\bigoplus_{i_1<\cdots<i_j}Q\big[\frac{1}{x_{i_1}\cdots x_{i_j}}\big]\alpha_{i_1}\cdots \alpha_{i_j}$$
  where $\alpha_{i}^{2}=0$, $\alpha_i\alpha_j=-\alpha_j\alpha_i$, and $|\alpha_i|=1$. Then $(\cC(x_1,\dots,x_n), \sum_{i=1}^{n}\alpha_i)$ is a complex and it turns out that the $k$th homology of the total complex
  $$(\cC(x_1,\dots,x_n)\otimes_Q\Omega_{Q/\C}^\bu\uuu, -df+ud+\sum_{i=1}^{n}\alpha_i)$$
is isomorphic to the $k$th hypercohomology $\mathbb{H}_k R\Gamma_{\fm} (\Omega^\bu_{Q/\C}\uuu, -df+ud)$; see \cite[Chapter 7]{iyen1} for details.

Let $f\in Q=\mathbb{C}[x_1,\dots,x_n]$ and assume that $\fm = (x_1, \dots, x_n)$ 
is the only critical point of the map $f: \mathbb{A}^n_\C \to \mathbb{A}^1_\C$. Let
\begin{equation}\label{tw} 
  \begin{aligned}
  \cH_{f}^{(0)}&:=\rH_n(\Omega^{\bullet}_{Q/\C}\uuu, -d{f} +ud)\\
  \cH_{-f}^{(0)}&:=\rH_n(\Omega^{\bullet}_{Q/\C}\uuu, d{f} +ud)
  \end{aligned}
\end{equation}
where $u$ is a formal variable of degree 2.  Since $f$ has an isolated singularity, it is known that
$$\cH_{f}^{(0)}\cong (\Omega^{\bullet}_{Q/\C}\uuu, -d{f} +ud)\cong \frac{\Omega^{n}_{Q/\C}}{(-df\wedge)\Omega^{n-1}_{Q/\C}}\uuu.$$
Moreover, letting $Q_\fm$ be the localization of $Q$ at $\fm$, under the condition $\mathrm{Sing}(f)=\fm$ we have
\begin{equation}\label{se4211}
  (\Omega^{\bullet}_{Q/\C}\uuu, -d{f} +ud)\cong (\Omega^{\bullet}_{Q_\fm/\C}\uuu, -d{f} +ud).
\end{equation}
Since the local cohomology functor $R\Gamma_m$ is the right adjoint of the canonical inclusion functor, we have
$$(\Omega^{\bullet}_{Q_\fm/\C}\uuu, -d{f} +ud)\cong R\Gamma_m (\Omega^{\bullet}_{Q_\fm/\C}\uuu, -d{f} +ud).$$

\begin{definition} \label{def825}
Grothendieck's residue map $\Gres: \rH^{n}_{\fm}(\Omega^n_{Q_\fm/\C}) \to \C$ is the unique  $\C$-linear map 
such that 
\begin{equation} \label{E527}
\Gres \left[  \frac{dx_1 \cdots dx_n} {x_1^{a_1}, \cdots, x_n^{a_n}} \right] =
\begin{cases} 
1 & \text{if $a_i = 1$ for all $i$, and } \\
0 & \text{otherwise.}
\end{cases}  
\end{equation}
In particular, we will denote $\Gres\left[  \frac{\psi_1\psi_2dx_1 \cdots dx_{n}} {f_1, \dots, f_{n}} \right]$ by
$\GGres(\psi_1dx_1\cdots dx_{n},\psi_2dx_1\cdots dx_{n})$ where $f_i:=\frac{\partial f}{\partial x_i}$. For the $\Z/2$-graded $Q_\fm$-module $\Omega^\bu_{Q_\fm/\C}$, it induces a map $\overline{\Gres}: \rH^{n}_{\fm}(\Omega^n_{Q_\fm/\C}/d\Omega^{n-1}_{Q_\fm/\C}) \to \C$; see \cite[Section 5]{kunz}. By abuse of notation, letting its $\C\uuu$-linear extension be $\Gres$ again, we define
$$
\res=: \on{res}_{Q,\fm}: \mathbb{H}_{2n} R\Gamma_\fm(\Omega^\bu_{Q_\fm/\C}\uuu,ud)  \to \C\uuu
$$
by the composition of
$$
\mathbb{H}_{2n} R\Gamma_\fm(\Omega^\bu_{Q_\fm/\C}\uuu,ud)  
\to \rH_\fm^n(\Omega^n_{Q_\fm/\C}/d\Omega^{n-1}_{Q_\fm/\C})\uuu \xra{\Gres} \C\uuu.
$$
\end{definition}

\subsection{Higher residue pairings}\label{32hr}

Let $f_j:=\frac{\partial f}{\partial x_j}$. We define for $j=1,\dots,n$
\begin{equation}\label{eq31}
  \Phi_{j,f}(h(x_1,\dots,x_n)):=\sum_{k=0}^{\infty}a_{k,j,f}(x_1,\dots,x_n)u^k
  \end{equation}
where $a_{0,j,f}:=-\frac{h}{f_j}$ and $a_{n,j,f}:=\frac{1}{f_j}\frac{\partial}{\partial x_j}a_{n-1,j,f}$. More explicitly, one has
\begin{equation}\label{maineq35}
  \Phi_{j,f}(h)=\sum_{k=0}^{\infty}(\frac{1}{f_j}\frac{\partial}{\partial x_j})^k(-\frac{h}{f_j})u^k.
\end{equation}
We extend $\Phi_{j,f}$ $\C\uuu$-linearly. Hence, it makes sense to say
$$\Phi_{j,f}\circ\Phi_{i,f}(h).$$
In particular, we have the followings.

\begin{lemma}\label{lemma2}
$$\Phi_{j,f}\circ\Phi_{i,f}(h)=\Phi_{i,f}\circ\Phi_{j,f}(h).$$
\end{lemma}

\begin{proof}
From \eqref{maineq35}, we see that
  \begin{equation}
    \begin{aligned}
      \Phi_{j,f}\circ\Phi_{i,f}(h)&=\sum_{l=0}^{\infty}(\frac{1}{f_j}\frac{\partial}{\partial x_j})^l(-\frac{\Phi_{i,f}(h)}{f_j})u^l\\
      &=\sum_{k=0}^{\infty}\sum_{l+m=k}(\frac{1}{f_j}\frac{\partial}{\partial x_j})^l\Big( \frac{1}{f_{j}}(\frac{1}{f_i}\frac{\partial}{\partial x_i})^m
      (\frac{h}{f_i})\Big)u^{k}.
      \end{aligned}
  \end{equation}
  Observing
  $$\sum_{l+m=k}(\frac{1}{f_j}\frac{\partial}{\partial x_j})^l\Big( \frac{1}{f_{j}}(\frac{1}{f_i}\frac{\partial}{\partial x_i})^m
      (\frac{h}{f_i})\Big)=\sum_{l+m=k}(\frac{1}{f_i}\frac{\partial}{\partial x_i})^l\Big( \frac{1}{f_{i}}(\frac{1}{f_j}\frac{\partial}{\partial x_j})^m
  (\frac{h}{f_j}),$$
  we establish the proof.
  \end{proof}

\begin{lemma}\label{lemma3}
  We have 
  $$(-f_idx_i+udx_i\frac{\partial}{\partial x_i})\Phi_{i,f}(h)=hdx_i.$$
  In particular,
$$(-f_idx_i+udx_i\frac{\partial}{\partial x_i})\Phi_{i,f}(\Phi_{j,f}(h)))=\Phi_{j,f}(h)dx_i.$$
\end{lemma}

\begin{proof}
An easy computation establishes the proof.
  \end{proof}

\begin{lemma}\label{lem4}
  We have that
  $$h(x_1,\dots,x_n)dx_1\cdots dx_n\in \rH_n(\Omega^\bu_{Q_\fm/\C}\uuu, -df+ud)$$
  corresponds to $\omega_f(h)\in  \HH_n\big(\cC(x_1,\dots,x_n) \otimes_{Q_\fm}  (\Omega_{Q_\fm/\C}^\bu\uuu, -df+ud)\big)$ where 
$$\scalebox{0.9}{$\begin{aligned}
  \omega_f(h)&:=h(x_1,\dots,x_n)dx_1\cdots dx_n\\
  &+\sum_{i=1}^{n}(-1)^{i+1}\Phi_{i,f}(h)\alpha_i dx_1\wedge\cdots\wedge \widehat{dx_i}\wedge\cdots \wedge dx_n\\
  &+\sum_{j<i}^{n}(-1)^{i+j+3} \Phi_{j,f}\circ\Phi_{i,f}(h)\alpha_j\alpha_i dx_1\wedge\cdots\wedge \widehat{dx_j}\wedge\cdots\wedge \widehat{dx_i}\wedge\cdots \wedge dx_n\\
  &+\sum_{k<j<i}^{n}(-1)^{i+j+k+6}\Phi_{k,f}\circ\Phi_{j,f}\circ\Phi_{i,f}(h)\alpha_k\alpha_j\alpha_i dx_1\wedge\cdots\wedge \widehat{dx_k}\wedge\cdots\wedge \widehat{dx_j}\wedge\cdots\wedge \widehat{dx_i}\wedge\cdots \wedge dx_n\\
  &+\cdots+(-1)^{n(n+1)}\Phi_{1,f}\circ\cdots\circ \Phi_{n,f}(h)\alpha_1\cdots \alpha_n.
\end{aligned}$}$$
\end{lemma}

\begin{proof}
  It suffices to show that $\omega_f(h)$ is closed under $\sum_{i=1}^{n}\alpha_i-df\wedge+ud$.  A computation using Lemma~\ref{lemma2} and Lemma~\ref{lemma3} shows that
  $$\scalebox{0.8}
           {$\begin{aligned}
 (-df+ud)&\sum_{1\leq i_1<\cdots<i_k\leq n}(-1)^{\dagger} \Phi_{i_1,f}\circ\cdots\circ\Phi_{i_k,f}(h)\alpha_{i_1}\cdots\alpha_{i_k}
      dx_1\wedge\cdots\wedge \widehat{dx_{i_1}}\wedge\cdots\wedge \widehat{dx_{i_k}}\wedge\cdots \wedge dx_n\\
  =-(\sum_{i=1}^{n}\alpha_i)&\sum_{1\leq i_1<\cdots<i_{k-1}\leq n}(-1)^{\ddagger} \Phi_{i_1,f}\circ\cdots\circ\Phi_{i_{k-1},f}(h)\alpha_{i_1}\cdots\alpha_{i_{k-1}}
      dx_1\wedge\cdots\wedge \widehat{dx_{i_1}}\wedge\cdots\wedge \widehat{dx_{i_{k-1}}}\wedge\cdots \wedge dx_n
             \end{aligned}$}$$
  where $\dagger=i_1+\cdots+i_k+\frac{k(k+1)}{2}$ and $\ddagger=i_1+\cdots+i_{k-1}+\frac{(k-1)k}{2}$, which establishes the proof.
 \end{proof}

By taking all $\alpha_i=0$, Lemma~\ref{lem4} shows in \eqref{3.2} how to identify
  \begin{equation}\label{3.3}
    \rH_n(\Omega^\bu_{Q_\fm/\C}\uuu, -df+ud) \leftrightarrow \HH_n\big(\cC(x_1,\dots,x_n) \otimes_{Q_\fm}  (\Omega_{Q_\fm/\C}^\bu\uuu, -df+ud) \big).
\end{equation}
Moreover, from Lemma~\ref{lem4}, we see that the composition of maps in \eqref{3.2}
\begin{equation}\label{eq3.8}
 \cH_{f}^{(0)}\times \cH_{-f}^{(0)}\to \HH_{2n} \big(\cC(x_1,\dots,x_n) \otimes_{Q_\fm}  (\Omega_{Q_\fm/\C}^\bu\uuu, ud) \big) 
  \end{equation}
gives $\wedge\circ\mathsf{kun}(\omega_f(h),\omega_{-f}(g))$. For $i_1<\cdots<i_k$ and $j_1<\cdots<j_{n-k}$ such that $\{1,\dots,n\}/\{i_1,\dots,i_k\}=\{j_1,\dots,j_{n-k}\}$, we let
$$\begin{aligned}
  \alpha(\{i_1,\dots,i_k\})&:=\alpha_{i_1}\cdots\alpha_{i_k}\\
  \alpha(\{i_1,\dots,i_k\}^c)&:=\alpha_{j_1}\cdots\alpha_{j_{n-k}}.
\end{aligned}$$
We define $d\mathbf{x}(\{i_1,\dots,i_k\})$ and $d\mathbf{x}(\{i_1,\dots,i_k\}^c)$ in a similar way. Observe that
\begin{multline}\label{lem3935}
  \wedge\circ\mathsf{kun}(\omega_f(h),\omega_{-f}(g))\\
  =\frac{1}{2^n}\Big(
  h\cdot\Phi_{1,-f}\circ\cdots\circ \Phi_{n,-f}(g)dx_1\cdots dx_n\alpha_1\cdots\alpha_n\\
  +\Phi_{1,f}\circ\cdots\circ \Phi_{n,f}(h)\cdot g\alpha_1\cdots\alpha_ndx_1\cdots dx_n\\
  +\sum_{k=1}^{n-1}\sum_{i_1<\cdots<i_k}(-1)^{\ast}\Phi_{i_1,f}\circ\cdots\circ\Phi_{i_k,f}(h)\cdot\Phi_{j_1,-f}\circ\cdots\circ\Phi_{j_{n-k},-f}(g)\\
  \alpha(\{i_1,\dots,i_k\})d\mathbf{x}(\{i_1,\dots,i_k\}^c)\alpha(\{i_1,\dots,i_k\}^c)d\mathbf{x}(\{i_1,\dots,i_k\})\Big)\\
  =\frac{1}{2^n}\Big(
  (-1)^n h\cdot\Phi_{1,-f}\circ\cdots\circ \Phi_{n,-f}(g)+\Phi_{1,f}\circ\cdots\circ \Phi_{n,f}(h)\cdot g\\
  +\sum_{k=1}^{n-1}\sum_{i_1<\cdots<i_k}(-1)^{n-k}\Phi_{i_1,f}\circ\cdots\circ\Phi_{i_k,f}(h)\cdot\Phi_{j_1,-f}\circ\cdots\circ\Phi_{j_{n-k},-f}(g)
  \Big)\alpha_1\cdots\alpha_ndx_1\cdots dx_n
  \end{multline}
where $\ast=\frac{n(n+1)}{2}+\frac{k(k+1)}{2}+\frac{(n-k)(n-k+1)}{2}=k(n-k)$.

\begin{lemma}\label{lemm35com}
In $\HH_{2n} \big(\cC(x_1,\dots,x_n) \otimes_{Q_\fm}  (\Omega_{Q_\fm/\C}^\bu\uuu, ud) \big)$,  we have
  \begin{multline}
(-1)^{n-k}\Phi_{i_1,f}\circ\cdots\circ\Phi_{i_k,f}(h)\cdot\Phi_{j_1,-f}\circ\cdots\circ\Phi_{j_{n-k},-f}(g)\alpha_1\cdots\alpha_ndx_1\cdots dx_n\\
    =\textsf{boundary}+(-1)^nh\cdot\Phi_{1,-f}\circ\cdots\circ\Phi_{n,-f}(g)\alpha_1\cdots\alpha_ndx_1\cdots dx_n\\
     =\textsf{boundary}+\Phi_{1,f}\circ\cdots\circ\Phi_{n,f}(h)\cdot g\alpha_1\cdots\alpha_ndx_1\cdots dx_n.
  \end{multline}
  \end{lemma}

\begin{proof}
  A computation using Lemma~\ref{lemma3} shows that
  \begin{multline*}
    \Phi_{i_1,f}\circ\cdots\circ\Phi_{i_k,f}(h)\cdot\Phi_{j_1,-f}\circ\cdots\circ\Phi_{j_{n-k},-f}(g)\alpha_1\cdots\alpha_ndx_1\cdots dx_n\\
    =-\Phi_{i_1,f}\circ\cdots\circ\Phi_{i_k,f}\circ\Phi_{j_1,f}(h)\cdot\Phi_{j_2,-f}\circ\cdots\circ\Phi_{j_{n-k},-f}(g)\alpha_1\cdots\alpha_ndx_1\cdots dx_n\\
    +(ud+\sum_{i=1}^{n}\alpha_i)\big(\Phi_{i_1,f}\circ\cdots\circ\Phi_{i_k,f}\circ\Phi_{j_1,f}(h)\cdot\Phi_{j_1,-f}\circ\Phi_{j_2,-f}\circ\cdots\circ\Phi_{j_{n-k},-f}(g)\big)\\
    \alpha_1\cdots\alpha_ndx_1\cdots\widehat{dx_{j_1}}\cdots dx_n.
  \end{multline*}
  It implies
\begin{multline*}
(-1)^{n-k}\Phi_{i_1,f}\circ\cdots\circ\Phi_{i_k,f}(h)\cdot\Phi_{j_1,-f}\circ\cdots\circ\Phi_{j_{n-k},-f}(g)\alpha_1\cdots\alpha_ndx_1\cdots dx_n\\
    =\textsf{boundary}+\Phi_{1,f}\circ\cdots\circ\Phi_{n,f}(h)\cdot g\alpha_1\cdots\alpha_ndx_1\cdots dx_n.
  \end{multline*}
  A similar observation also shows the other equality.
\end{proof}

From \eqref{lem3935} and Lemma~\ref{lemm35com}, \eqref{eq3.8} becomes
\begin{multline*}(hdx_1\cdots dx_n,gdx_1\cdots dx_n)\\
  \to\frac{1}{2}\big(\Phi_{1,f}\circ\cdots\circ \Phi_{n,f}(h)\cdot g+(-1)^nh\cdot\Phi_{1,-f}\circ\cdots\circ \Phi_{n,-f}(g)\big)\alpha_1\cdots\alpha_ndx_1\cdots dx_n.
\end{multline*}
For $h\in\mathbb{C}[x_1,\dots,x_n]$, letting
\begin{equation}\label{eq36}
  \Phi_{1,f}\circ\cdots\circ \Phi_{n,f}(h)=\sum_{i=0}^{\infty}b_{i,f}(h)u^i,
  \end{equation}
we may see that the composition of maps in \eqref{3.2} on $\cH_{f}^{(0)}\times \cH_{f}^{(0)}$ is given by
\begin{equation}\label{hres1}
    \begin{aligned}
  \sH_f(hdx_1\cdots dx_n,gdx_1\cdots dx_n)&:=\frac{1}{2}\sum_{i=0}^{\infty}\Big(\res((-1)^nb_{i,f}(h)g\alpha_1\cdots\alpha_ndx_1\cdots dx_n)\\
      &+\res(b_{i,-f}(g)h\alpha_1\cdots\alpha_ndx_1\cdots dx_n)\Big) u^{i}.
  \end{aligned}
  \end{equation}
For $\omega_1,\omega_2\in \cH_{f}^{(0)}$, we extend this pairing $\mathbb{C}\uuu$-sesquilinearly. Letting $d\mathbf{x}:=dx_1\cdots dx_n$ and 
\begin{equation}\label{impo2}
  \sH_f^{i}(hd\mathbf{x},gd\mathbf{x}):=\frac{1}{2}\Big(\Gres((-1)^nb_{i,f}(h)gd\mathbf{x})+\Gres(b_{i,-f}(g)hd\mathbf{x})\Big),
  \end{equation}
one can rewrite the composition of maps in \eqref{3.2}, called higher residue  pairings, as
\begin{equation}\label{hres2}
  \sH_f(hd\mathbf{x},gd\mathbf{x}):=\sum_{i\geq0}\sH^{i}_{f}(hd\mathbf{x},gd\mathbf{x})u^{i}.
  \end{equation}
We note that
\begin{equation}\label{impotant}
  \begin{aligned}
  b_{0,f}(h)&=(-1)^{n}\frac{h}{f_1\cdots f_n}\\
  b_{1,f}(h)&= (-1)^n\sum_{i=1}^{n}\frac{1}{f_1\cdots f_{i-1}}\cdot\frac{1}{f_i}\frac{\partial}{\partial x_i}(\frac{h}{f_i\cdots f_n})\\
  &=(-1)^n\sum_{i=1}^{n}\frac{h_i}{f_i(f_1\cdots f_n)}+(-1)^{n+1}\sum_{i=1}^{n}\sum_{j=i}^{n}\frac{hf_{ji}}{f_if_j(f_1\cdots f_n)}\\
 b_{2,f}(h)&= (-1)^n\sum_{i=1}^{n}\frac{1}{f_1\cdots f_{i-1}}\cdot(\frac{1}{f_i}\frac{\partial}{\partial x_i})^2(\frac{h}{f_i\cdots f_n})\\
 &+(-1)^n\sum_{i=1}^{n}\sum_{j=1}^{i-1}\frac{1}{f_1\cdots f_{i-j-1}}\cdot\frac{1}{f_{i-j}}\frac{\partial}{\partial x_{i-j}}
 \Big(\frac{1}{f_{i-j}\cdots f_{i-1}}\cdot\frac{1}{f_i}\frac{\partial}{\partial x_i}(\frac{h}{f_i\cdots f_n})\Big).
  \end{aligned}
\end{equation}
In general, the formula for $b_{k,f}(h)$ is given as follows:
\begin{multline}\label{important2}
  b_{k,f}(h)=(-1)^n\sum_{(a_1,\dots,a_n)}(\frac{\partial}{f_1\partial x_1})^{a_1}\frac{1}{f_1}\cdots(\frac{\partial}{f_k\partial x_k})^{a_k}\frac{1}{f_k}\cdots(\frac{\partial}{f_n\partial x_n})^{a_n}\frac{h}{f_n}
\end{multline} 
where the sum is over all $(a_1,\dots,a_n)$ such that $a_1+\cdots+a_n=k$ with non-negative integers $a_i$.

\section{Equivalence between the canonical  pairing and  $(-1)^{\frac{n(n+1)}{2}}\sH_f$.}\label{eqcanhres}

\subsection{Characteristic equation of the canonical  pairing}\label{defcan}

For a smooth and proper $\C$-linear dg category $\cA$, the canonical  pairing $\langle-,-\rangle_{HP(\cA)\otimes HP(\cA^\op)}$ on  $HP(\cA)\otimes HP(\cA^\op)$ can be seen as the inverse to the Chern character $\Ch_{HP}(\Delta_\cA)\in HP(\cA^\op\otimes\cA)$ where $\Delta_\cA\in \Perf(\cA^\op\otimes\cA)$ is the diagonal functor $\Delta_\cA:\cA\otimes\cA^\op\to\Perf(\C)$ given by $\Delta_\cA(E\otimes F^\vee)=\Hom_\cA(F,E)$. More specifically, we have
\begin{equation}\label{decan}
  \big(\langle-,-\rangle_{HP(\cA)\otimes HP(\cA^\op)}\otimes\id_{HP(\cA)}\big)\circ\big(\gamma\otimes\Ch_{HP}(\Delta_\cA)\big)=\gamma\text{ for all }\gamma\in HP_\ast(\cA);
  \end{equation}
see \cite[page 1875]{PV: HRR} and \cite[page 25]{Shk: HRR}. Via the K\"unneth isomorphism,  after writing 
 \[\Ch_{HP}(\Delta_\cA) = \sum_i T^i \ot T_i \text{ for some } T^i \in HP_\ast (\cA ^{op} )\text{ and } \ T_i \in HP_\ast (\cA ) , \]
 it is known that $\lan-,-\ran _{HP(\cA)\otimes HP(\cA^\op)}$ is a unique non-degenerate bilinear pairing $\lan-,-\ran$ satisfying 
\begin{equation}\label{eqn: char pairing}   
\sum _i \lan \gamma , T^i \ran \lan T_i, \gamma ' \ran = \lan \gamma, \gamma ' \ran  \text{ for every } \gamma \in HP_\ast(\cA )\text{ and }  \gamma '\in HP_\ast (\cA ^{op});
\end{equation}
see \cite[Proposition 4.2]{Shk: HRR} and \cite[Section 2.5]{Kim}. We call Equation~\eqref{eqn: char pairing} the characteristic equation of the canonical  pairing on $HP(\cA)\otimes HP(\cA^\op)$ and use Equation~\eqref{eqn: char pairing} to check whether a given explicit formula for a pairing is the canonical  pairing $\lan-,-\ran_{\ccan}$ on  $\cH_{f}^{(0)}\times \cH_{f}^{(0)}$ or not.

\subsection{The properties of Grothendieck's residue}\label{gropro} 

From \cite{HRD,kunz}, we recall two properties of Grothendieck's residue; Transitivity law and Transformation law: 
\begin{equation}\label{grore}
\begin{aligned}
   \Gres&\left[  \frac{d\mathbf{x}\wedge d\mathbf{y}}{f_1(\mathbf{x})^{m_1}, \cdots, f_n(\mathbf{x})^{m_n},f_1(\mathbf{y}),\cdots,f_n(\mathbf{y})} \right]\\
   &=\Gres\left[  \frac{ \Gres\left[  \frac{d\mathbf{x}}
         {f_1(\mathbf{x})^{m_1}, \cdots, f_n(\mathbf{x})^{m_n}} \right]\wedge d\mathbf{y}}
     {f_1(\mathbf{y}),\cdots,f_n(\mathbf{y})} \right]\\
  &=\Gres\left[  \frac{(-1)^{n+\flat}d(\mathbf{y}-\mathbf{x})\wedge d\mathbf{y}}
    {\big(f_1(\bfy)-f_1(\bfx)\big)^{m_1},\cdots,\big(f_n(\bfy)-f_n(\bfx)\big)^{m_n}, f_1(\mathbf{y}), \cdots, f_n(\mathbf{y})} \right]
\end{aligned}
\end{equation}
where  $\flat=m_1+\cdots+m_n$, $m_i\geq 1$, and $d\mathbf{x}=dx_1\wedge\cdots\wedge dx_n$.

\subsection{The equivalence}\label{33eq}
From \cite{KR}; also see \cite[Proposition 4.1.1]{PV: HRR}, we know
$$\ch_{HH}(\Delta_{Q_\fm})=(-1)^{\frac{n(n+1)}{2}}\cdot\det(\Delta_j(\partial_i f))d\mathbf{x}\ot d\mathbf{y}\in
\rH_{2n}(\Omega^{\bullet}_{Q_\fm\otimes Q_\fm},  -d\widetilde{f});$$
where  $\widetilde{f}=f(\mathbf{y})-f(\mathbf{x})$ and
$$\Delta_j f:=\frac{f(x_1,\ldots,x_{j-1},y_j,y_{j+1}\ldots,y_n)-f(x_1,\ldots,x_{j-1},x_j,y_{j+1},\ldots,y_n)}{y_j-x_j}.$$
We note that there is a formula for  $\ch_{HN}(\Delta_{Q_\fm})$ in \cite{BW} by using a connection $\nabla$. Since $Q_\fm$ is a local ring, we may take a Levi-Civita connection, i.e., $\nabla^2=0$; see \cite[Remark 5.11]{BW} for the definition. Thus, one still has
$$\ch_{HN}(\Delta_{Q_\fm})=(-1)^{\frac{n(n+1)}{2}}\cdot\det(\Delta_j(\partial_i f))
d\mathbf{x}\ot d\mathbf{y}\in\rH_{2n}(\Omega^{\bullet}_{Q_\fm\otimes Q_\fm}\uuu,  ud-d\widetilde{f});$$
see \cite{BW, CKK, TK} for details. Similar to \cite[Proposition 4.1.2]{PV: HRR}, we have the followings.

\begin{theorem}\label{main}
  Let $\bde(\mathbf{x},\mathbf{y}):=\det\big(\Delta_j(\partial_i f)\big)$ and
  $[\sum_{i=0}^{\infty}g_i(\mathbf{x})u^id\mathbf{x}]\in \cH_{f}^{(0)}$. Under the assumption that the only singularity of $f\in Q$ is $\fm$, on  $\cH_{f}^{(0)}\times \cH_{f}^{(0)}$ we have
 \begin{equation}\label{A-de-eq}
(\sH_f(\cdot,(-1)^n\cdot)\ot\id)\circ([\sum_{i=0}^{\infty}g_i(\mathbf{x})u^id\mathbf{x}]\ot[(-1)^n\bde(\mathbf{x},\mathbf{y})d\bfx \ot d\bfy])=[\sum_{i=0}^{\infty}g_i(\mathbf{y})u^id\mathbf{y}].
 \end{equation}
 That is, $(-1)^{\frac{n(n+1)}{2}}\sH_f(\cdot,\cdot)$ is the canonical pairing $\langle-,-\rangle_{\ccan}$ on $\cH_{f}^{(0)}\times \cH_{f}^{(0)}$.
\end{theorem}

\begin{proof}
From the non-degeneracy of the residue  pairing $\GGres$ and the sesquilinearity with respect to $u$, it suffices to prove that for all $h\in\mathbb{C}[\mathbf{y}]$, we have
  \begin{multline}
    \sum_{i=0}^{\infty}\GGres\big(h(\mathbf{y})d\bfy,(\sH_f(\cdot,(-1)^n\cdot)\ot\id)\circ(g_i(\mathbf{x})d\mathbf{x}\ot(-1)^n\bde(\mathbf{x},\mathbf{y})d\bfx\otimes d\bfy)\big) u^i\\
    =\sum_{i=0}^{\infty}\GGres\big(h(\mathbf{y})d\mathbf{y},g_i(\mathbf{y})d\mathbf{y}\big)u^i.
    \end{multline}
  That is, we need to prove that for all $h\in\mathbb{C}[\mathbf{y}]$
  \begin{equation}\label{hler}
    \GGres\Big(h(\mathbf{y})d\mathbf{y},\big(\sH_f(\cdot,(-1)^n\cdot)\ot\id)\circ(g(\mathbf{x})d\mathbf{x}\ot(-1)^n\bde(\mathbf{x},\mathbf{y})d\bfx\ot d\bfy\big)\Big)=\GGres\big(h(\mathbf{y})d\mathbf{y},g(\mathbf{y})d\mathbf{y}\big).
\end{equation}
From Lemma~\ref{lemm35com}, we see that on $\cH_{f}^{(0)}\times \cH_{f}^{(0)}$ 
\begin{equation}\label{newres1}
   \sH_f(hdx_1\cdots dx_n,gdx_1\cdots dx_n)=\sum_{i=0}^{\infty}\Gres\big((-1)^nb_{i,f}(h)gdx_1\cdots dx_n\big)u^{i}.
  \end{equation}
Using Formula~\eqref{newres1} we see that the left side of Equation~\eqref{hler} becomes
 $$\sum_{i=0}^{\infty}\Gres \left[  \frac{h(\bfy)  \Gres \left[ (-1)^nb_{i,f}(g(\bfx))\bde(\bfx,\bfy)d\bfx\right]u^{i}d\bfy    }{f_1(\bfy),\cdots, f_n(\bfy)} \right].$$
 For $i=0$, from \cite[Proposition 4.1.2]{PV: HRR} we already know that 
$$\Gres \left[  \frac{h(\bfy)  \Gres \left[(-1)^n b_{0,f}(g(\bfx))\bde(\bfx,\bfy)d\bfx\right]d\bfy    }{f_1(\bfy),\cdots, f_n(\bfy)} \right]
  =\Gres \left[  \frac{h(\bfy)g(\bfy)d\bfy} {f_1(\bfy),\cdots, f_n(\bfy)} \right].$$
Thus, we need to show that for $i\geq1$,
\begin{equation}\label{mmmm}
  \Gres \left[  \frac{h(\bfy)  \Gres \left[(-1)^nb_{i,f}(g(\bfx))\bde(\bfx,\bfy)d\bfx\right]d\bfy    }
    {f_1(\bfy),\cdots, f_n(\bfy)} \right]=0.
  \end{equation}
Now observe that from Formula~\eqref{important2}, for $g\in \C[\bfx]$ we see that $b_{i,f}(g)$ with $i\geq 1$ is the summation of the followings
$$\frac{\Psi(g,f)(\bfx)}{f_1^{m_1}\cdots f_n^{m_n}}$$
where $\Psi$ consists of the derivatives of $f$ and $g$ and the exponents in the denominator are bigger than or equal to $1$ and in particular at least some $m_j>1$.  Using the transitivity law and the transformation law of Grothendieck's residue in \eqref{grore}, we see the followings:
\begin{multline*}
\Gres \left[  \frac{h(\bfy)  \Gres \left[ (-1)^nb_{i,f}(g(\bfx))\bde(\bfx,\bfy)d\bfx\right]d\bfy    }
  {f_1(\bfy),\cdots, f_n(\bfy)} \right]\\
=\Gres \left[  \frac{(-1)^{\flat}h(\bfy)  \Psi(g,f)(\bfx)\bde(\bfx,\bfy)d(\bfy-\bfx) d\bfy    }
  {(f_1(\bfy)-f_1(\bfx))^{m_1},\cdots, (f_n(\bfy)-f_n(\bfx))^{m_n},f_1(\bfy),\cdots, f_n(\bfy)} \right]\\
=\Gres \left[  \frac{(-1)^{\flat}h(\bfy)  \Gres_{\bfx=\bfy} \left[ \frac{\Psi(g,f)(\bfx)\bde(\bfx,\bfy)d(\bfy-\bfx)}{(f_1(\bfy)-f_1(\bfx))^{m_1},\cdots, (f_n(\bfy)-f_n(\bfx))^{m_n}} \right]d\bfy    }
  {f_1(\bfy),\cdots,f_n(\bfy)} \right]\\
=\Gres \left[  \frac{(-1)^{\flat}h(\bfy)  \Gres_{\bfx=\bfy} \left[ \frac{\Psi(g,f)(\bfx)d(\bfy-\bfx)}{(f_1(\bfy)-f_1(\bfx))^{m_1-1},\cdots, (f_n(\bfy)-f_n(\bfx))^{m_n-1},y_1-x_1,\cdots,y_n-x_n} \right]d\bfy    }
  {f_1(\bfy),\cdots, f_n(\bfy)} \right]
\end{multline*}
where $\flat=m_1+\cdots+m_n$. Since  at least some $m_j>1$, we see that
$$\Gres_{\bfx=\bfy} \left[ \frac{\Psi(g,f)(\bfx)d(\bfy-\bfx)}{(f_1(\bfy)-f_1(\bfx))^{m_1-1},\cdots, (f_n(\bfy)-f_n(\bfx))^{m_n-1},y_1-x_1,\cdots,y_n-x_n} \right]=0,$$
which establishes \eqref{mmmm}. Hence,
$(-1)^{\frac{n(n-1)}{2}}\sH_f(\cdot,(-1)^n\cdot)$ is the canonical pairing on $\cH_{f}^{(0)}\times \cH_{f}^{(0)}$.
\end{proof}

\begin{remark}
In \cite{HK1}, we consider the composition of the global counterparts of maps in \eqref{3.2} and establish that it is  a canonical pairing.
\end{remark}

\section{Saito's higher residue  pairings $\Kw$ and $u^n\sH_f$.}

\subsection{Equivalence}\label{sec4}
Recall that Saito's higher residue  pairings $\sK_f:\rH_n(\Omega^{\bullet}_{Q/\C}\uu, -d{f} +ud)\times\rH_n(\Omega^{\bullet}_{Q/\C}\uu, -d{f} +ud)\to \mathbb{C}\uu$ satisfy
\begin{enumerate}[label=\arabic*)]
\item $\mathbb{C}\uu$-sesquilinear, i.e., $\sK_f(u\omega_1,\omega_2)=\sK_f(\omega_1,-u\omega_2)=u\sK_f(\omega_1,\omega_2)$
\item
  $\sK_f(\omega_1,\omega_2)=(-1)^n\sK_f(\omega_2,\omega_1)^\star$ where $\lambda(u)^\star:=\lambda(-u)\in\C\uu$
\item  $\partial_u\sK_f(\omega_1,\omega_2)=\sK_f(\nabla_{\partial_u}^{f}\omega_1,\omega_2)-\sK_f(\omega_1,\nabla_{\partial_u}^{f}\omega_2)$ with $\nabla_{\partial_u}^{f}=\partial_u+\frac{f}{u^2}$
\item the restriction of $\sK_f$ to $\cH_{f}^{(0)}$ takes a value in $\mathbb{C}\uuu u^{n}$ 
\item for $\omega_i=\psi_idx_1 \cdots dx_{n}+\sum_{k=1}^{\infty}\omega_{i,k}u^k$ with $i=1,2$,
$$\sK_f(\omega_1,\omega_2)=u^n\Gres\left[  \frac{\psi_1\psi_2dx_1 \cdots dx_{n}} {f_1, \dots, f_{n}} \right]+O(u^{n+1});$$
\end{enumerate}
see \cite{her3,He3,shkl} for details. On $\cH_f^{(0)}$, Saito's higher residue  pairings $\sK_f$ can be written as
$$\Kw(\omega_1,\omega_2):=\sum_{i=0}^{\infty}K_{f}^{i}(\omega_1,\omega_2)u^{n+i}\text{ where }K_{f}^{i}(\omega_1,\omega_2)\in \mathbb{C}.$$
The characterization of higher residue  pairings on $\cH_f^{(0)}$ is that they are unique if they satisfy  $1)$ to $5)$ and the flatness conditions in Lemma~\ref{imlem55} and Lemma~\ref{lem57}. We remark that without checking $5)$ the uniqueness holds up to a constant; see \cite[Theorem 5.1]{tsaito}, \cite[Appendix]{msai2}, and \cite{shkl} for details. Indeed, Shklyarov's conjecture is about the determination of the ambiguity of the constant with regard to checking $5)$ in the process of the comparison of the two higher residue pairings in \cite{shkl}. The explicit formula  \eqref{hres1} facilitates to check $5)$ and the other properties directly, which leads to the resolution of the indeterminacy in the comparison; see Theorem~\ref{saitomain}. We also note that this strategy is the generalization of the work about the canonical pairing on the Hochschild homology of the category of matrix factorizations by A. Polishchuk and A. Vaintrob \cite{PV: HRR}; see \cite[Section 1.2]{shkl} for details and compare with \cite{BW3,Kim,HK1}. 

By construction it is clear that 
\begin{itemize}
\item
$\Cw$ is  $\mathbb{C}\uu$-sesquilinear
\item
$\Cw^{0}(\omega_1,\omega_2)=\GGres(\omega_1,\omega_2)$. 
\end{itemize}
Since $b_{i,-f}(g)=(-1)^{n+i}b_{i,f}(g)$ from \eqref{important2}, Equation~\eqref{impo2} implies that $\Cw^{i}$ is $(-1)^i$-symmetric., i.e.,
$\Cw^{i}(hd\mathbf{x},gd\mathbf{x})=(-1)^i\Cw^{i}(gd\mathbf{x},hd\mathbf{x})$, which implies that
$$\sum_{i=0}^{n}\Cw^{i}(hd\mathbf{x},gd\mathbf{x})u^i=(\sum_{i=0}^{n}\Cw^{i}(gd\mathbf{x},hd\mathbf{x})u^i)^\star.$$
From this and the $\C\uu$-sesquilinearity of $\Cw$, in general for $\omega_1,\omega_2\in\cH_f^{(0)}$ we see that
 $$u^n\Cw(\omega_1,\omega_2)=(-1)^n\big(u^n\Cw(\omega_2,\omega_1)\big)^\star.$$

\begin{lemma}\label{lem55}
  For $\omega_i\in\cH_f^{(0)}$ let $\CCw(\omega_1,\omega_2):=\sum_{k=0}^{\infty}\Cw^{k}(\omega_1,\omega_2)u^{n+k}$. Then
  $$\CCw(\bbco\omega_1,\omega_2)-\CCw(\omega_1,\bbco\omega_2)=\frac{d}{du}\CCw(\omega_1,\omega_2)\text{ where }\bbco=\frac{d}{du}+\frac{f}{u^2}.$$
  \end{lemma}

\begin{proof}
 Let $\omega_1=\sum_{i=0}^{\infty}\omega_{1,i}u^i$ and $\omega_2=\sum_{j=0}^{\infty}\omega_{2,j}u^j$ where $\omega_{1,i},\omega_{2,j}\in\Omega^{n}_{Q/\C}$. By the $\mathbb{C}\uu$-sesquilinearity, we observe that
\begin{multline}
  \frac{d}{du}\sum_{k=0}^{\infty}\Cw^{k}(\sum_{i=0}^{\infty}\omega_{1,i}u^i,\sum_{j=0}^{\infty}\omega_{2,j}u^j)u^{k+n}
  =\frac{d}{du}\sum_{l=0}^{\infty}\sum_{l=i+j+k}(-1)^j\Cw^{k}(\omega_{1,i},\omega_{2,j})u^{l+n}\\
  =\sum_{l=0}^{\infty}\sum_{l=i+j+k}(-1)^j\big(\Cw^{k}(i\omega_{1,i},\omega_{2,j})+\Cw^{k}(\omega_{1,i},j\omega_{2,j})\big)u^{l+n-1}\\
  +\sum_{l=0}^{\infty}\sum_{l=i+j+k}(-1)^j(k+n)\Cw^{k}(\omega_{1,i},\omega_{2,j})u^{l+n-1}.
\end{multline}
In particular, we see that 
\begin{multline}
  \sum_{l=0}^{\infty}\sum_{l=i+j+k}(-1)^j\big(\Cw^{k}(i\omega_{1,i},\omega_{2,j})+\Cw^{k}(\omega_{1,i},j\omega_{2,j})\big)u^{l+n-1}\\
  =\sum_{k=0}^{\infty}\big(\Cw^{k}(\sum_{i=1}^{\infty}i\omega_{1,i}u^{i-1},\sum_{j=0}^{\infty}\omega_{2,j}u^j)
  -\Cw^{k}(\sum_{i=0}^{\infty}\omega_{1,i}u^i,\sum_{j=1}^{\infty}j\omega_{2,j}u^{j-1})\big)u^{k+n}\\
  =\sum_{k=0}^{\infty}\Cw^{k}(\frac{d}{du}\sum_{i=0}^{\infty}\omega_{1,i}u^i,\sum_{j=0}^{\infty}\omega_{2,j}u^j)u^{k+n}
  -\sum_{k=0}^{\infty}\Cw^{k}(\sum_{i=0}^{\infty}\omega_{1,i}u^i,\frac{d}{du}\sum_{j=0}^{\infty}\omega_{2,j}u^j)u^{k+n}.
\end{multline}
On the other hand, we see that
\begin{multline}
\sum_{k=0}^{\infty}\Cw^{k}(\sum_{i=0}^{\infty}f\omega_{1,i}u^{i-2},\sum_{j=0}^{\infty}\omega_{2,j}u^j)u^{k+n}
-\sum_{k=0}^{\infty}\Cw^{k}(\sum_{i=0}^{\infty}\omega_{1,i}u^{i},\sum_{j=0}^{\infty}f\omega_{2,j}u^{j-2})u^{k+n}\\
=\sum_{l=0}^{\infty}\sum_{l=i+j+k}(-1)^j\big(\Cw^{k}(f\omega_{1,i},\omega_{2,j})-\Cw^{k}(\omega_{1,i},f\omega_{2,j})\big)u^{l+n-2}.
\end{multline}
Hence, in order to show
 $$\CCw(\bbco\omega_1,\omega_2)-\CCw(\omega_1,\bbco\omega_2)=\frac{d}{du}\CCw(\omega_1,\omega_2)\text{ where }\bbco=\frac{d}{du}+\frac{f}{u^2},$$
it suffices to show that
\begin{align}
  \Cw^{0}(f\omega_{1,0},\omega_{2,0})-\Cw^{0}(\omega_{1,0},f\omega_{2,0})&=0\label{aimpo1}\\
  \Cw^{k+1}(f\omega_{1,i},\omega_{2,j})-\Cw^{k+1}(\omega_{1,i},f\omega_{2,j})&=(k+n)\Cw^{k}(\omega_{1,i},\omega_{2,j})\text{ for }k\geq0\label{aimpo2}.
 \end{align}

\eqref{aimpo1} is clear from the definition of $\Cw^{0}$ in  \eqref{impo2} and \eqref{impotant}. Note that by some calculation using \eqref{impotant} and \eqref{impo2}, we easily see that
$$\Cw^{1}(f\omega_{1,i},\omega_{2,j})-\Cw^{1}(\omega_{1,i},f\omega_{2,j})=n\Cw^{0}(\omega_{1,i},\omega_{2,j}).$$
For general $k\geq 1$ for \eqref{aimpo2}, applying the Leibniz rule to the formula of $b_{k+1,f}$ in  \eqref{important2} and $\frac{\partial}{f_i\partial x_i}f=1$, we see that there is a nonnegative integer $\mathsf{C}$ satisfying
$$b_{k+1,f}(f\omega_{1,i})-fb_{k+1,f}(\omega_{1,i})=\mathsf{C}\cdot b_{k,f}(\omega_{1.i}).$$
Since the number of terms in $b_{k+1,f}$ in the summation is the number of solutions $(a_1,\dots,a_n)$ of $a_1+\cdots+a_n=k+1$ with non-negative integers $a_i$, it is $n+k \choose k+1$ by combination by repetition. Thus, using  $\frac{\partial}{f_i\partial x_i}f=1$ and the Leibniz rule again, we deduce that $\mathsf{C}$ must satisfy
$$(k+1){n+k \choose k+1}=\mathsf{C}\cdot{n+k-1\choose k}.$$
With a similar observation for $b_{k+1,-f}$ and $b_{k,-f}$, we conclude that
\begin{equation}\label{impo3}
  \begin{aligned}
  b_{k+1,f}(f\omega_{1,i})-fb_{k+1,f}(\omega_{1,i})&=(k+n)b_{k,f}(\omega_{1.i})\\
  fb_{k+1,-f}(\omega_{2,j})-b_{k+1,-f}(f\omega_{2,j})&=(k+n)b_{k,-f}(\omega_{2.j}).
  \end{aligned}
\end{equation}
Clearly, \eqref{aimpo2} follows from \eqref{impo3} and \eqref{impo2}.
\end{proof}

Let $\eta_1,\dots,\eta_\mu$ be a basis of
$$Q/(f_1,\dots,f_n)=\mathbb{C}[x_1,\dots,x_n]/(f_1,\dots,f_n).$$
Consider
\begin{equation}
  \begin{aligned}
  \cH_{f,\eta}^{(0)}&:=\rH_n(\Omega^{\bullet}_{Q/\C}[z_1,\dots,z_\mu]\uuu, -d{f}-\sum_{i=1}^{\mu}z_id\eta_i +ud)\\
  \cH_{-f,-\eta}^{(0)}&:=\rH_n(\Omega^{\bullet}_{Q/\C}[ z_1,\dots,z_\mu]\uuu, d{f}+\sum_{i=1}^{\mu}z_id\eta_i+ud).
  \end{aligned}
\end{equation}
Let $F:=f(x_1,\dots,x_n)+\sum_{i=1}^{\mu}z_i\eta_i(x_1,\dots,x_n)$. For $h,g\in\mathbb{C}[x_1,\dots,x_n][z_1,\dots,z_\mu]$, by replacing $f$ with $F$ in Equation~\eqref{hres1}, the extended higher residue  pairings on $\cH_{f,\eta}^{(0)} \times \cH_{f,\eta}^{(0)}$ is given by 
\begin{equation}\label{imfor57}
  \begin{aligned}
  \Cwn(hd\mathbf{x},gd\mathbf{x}):&=\frac{1}{2}\sum_{k=0}^{\infty}\Big((-1)^n\res\big(b_{k,F}(h)g\alpha_1\cdots\alpha_ndx_1\cdots dx_n\big)\\
  &+\res\big(b_{k,-F}(g)h\alpha_1\cdots\alpha_ndx_1\cdots dx_n\big)\Big) u^{k}.
  \end{aligned}
  \end{equation}
We also let
$$\Cwn^{k}(hd\mathbf{x},gd\mathbf{x}):=\frac{1}{2}\Big((-1)^n\Gres(b_{k,F}(h)gd\mathbf{x})+\Gres(b_{k,-F}(g)hd\mathbf{x})\Big)$$
and we note that $\Cwn^k$ is $\C[z_1,\dots,z_\nu]$-valued. The formula for $b_{k,F}$ is given as follows:
\begin{multline}\label{important3}
  b_{k,F}(h)=(-1)^n\sum_{(a_1,\dots,a_n)}(\frac{\partial}{F_1\partial x_1})^{a_1}\frac{1}{F_1}\cdots(\frac{\partial}{F_k\partial x_k})^{a_k}\frac{1}{F_k}\cdots(\frac{\partial}{F_n\partial x_n})^{a_n}\frac{h}{F_n}
\end{multline}
where $F_i:=\frac{\partial F}{\partial x_i}$ and the sum is over all $(a_1,\dots,a_n)$ such that $a_1+\cdots+a_n=k$ with non-negative integers $a_i$.

For $\omega_1,\omega_2\in\cH_{f,\eta}^{(0)}$ letting $\CCwn(\omega_1,\omega_2):=\sum_{k=0}^{\infty}\Cwn^{k}(\omega_1,\omega_2)u^{n+k}$ and replacing $f$ with $F$ in Lemma~\ref{lem55}, we have the followings.
\begin{lemma}\label{imlem55}
  $$\CCwn(\bbcon\omega_1,\omega_2)-\CCwn(\omega_1,\bbcon\omega_2)=\frac{d}{du}\CCwn(\omega_1,\omega_2)\text{ where }\bbcon=\frac{d}{du}+\frac{F}{u^2}.$$
  \end{lemma}

\begin{lemma}\label{lem57}
 For $i=1,\dots,\mu$, we have
 \begin{equation}\label{lem57eq}
   \CCwn(\bbcoo\omega_1,\omega_2)+\CCwn(\omega_1,\bbcoo\omega_2)=\frac{\partial}{\partial z_i}\CCwn(\omega_1,\omega_2)\text{ where }
   \bbcoo=\frac{\partial}{\partial z_i}-\frac{\eta_i}{u}.
   \end{equation}
  \end{lemma}

\begin{proof}
  From \eqref{imfor57}, \eqref{important3}, the Leibniz rule, and the sesquilinear extension of $\C\uu$, without loss of generality it suffices to show that for $h,g\in\mathbb{C}[x_1,\dots,x_n]$, 
  \begin{align}
    \Cwn^{0}(\eta_i hd\mathbf{x},gd\mathbf{x})-\Cwn^{0}(hd\mathbf{x},\eta_i gd\mathbf{x})&= 0\\
    \Cwn^{k+1}(\eta_i hd\mathbf{x},gd\mathbf{x})-\Cwn^{k+1}(hd\mathbf{x},\eta_i gd\mathbf{x})&= -\frac{\partial}{\partial z_i}\Cwn^k(hd\mathbf{x},gd\mathbf{x})\text{ for }k\geq0.\label{lem53512}
  \end{align}
  From the formula in \eqref{important3}, we see that
  \begin{multline*}
    \Cwn^{0}(\eta_i hd\mathbf{x},gd\mathbf{x})-\Cwn^{0}(hd\mathbf{x},\eta_i gd\mathbf{x})\\
    = \Gres(\frac{\eta_ihg}{F_1\cdots F_n}d\mathbf{x})-\Gres(\frac{\eta_ihg}{F_1\cdots F_n}d\mathbf{x})=0.
  \end{multline*}
For \eqref{lem53512}, we observe the followings. Using the Leibniz law and 
$$\frac{\partial }{\partial z_i}F_j=\frac{\partial }{\partial x_j}\eta_i,$$
we observe that for $\Theta(x_1,\dots,x_n,z_1,\dots,z_\mu)\in\C[x_1,\dots,x_n][z_1,\dots,z_\mu]$,
\begin{multline}\label{app51}
  \frac{\partial}{\partial z_i}\Big((\frac{\partial}{F_k\partial x_k})^{a_k}
  (\frac{1}{F_{k}}\cdot \Theta)\bigg)\\
  =\eta_i(\frac{\partial}{F_k\partial x_k})^{a_k+1}(\frac{1}{F_{k}}\cdot \Theta)
  -(\frac{\partial}{F_k\partial x_k})^{a_k+1}\eta_i(\frac{1}{F_{k}}\cdot\Theta)
  +(\frac{\partial}{F_k\partial x_k})^{a_k}(\frac{1}{F_{k}}\cdot \frac{\partial\Theta}{\partial z_i}).
\end{multline}
Using \eqref{app51}, we deduce that for $h\in\C[x_1,\dots,x_n]$ and nonnegative integers $a_k$,
\begin{multline}\label{app52}
  \frac{\partial}{\partial z_i}\bigg((\frac{\partial}{F_1\partial x_1})^{a_1}\frac{1}{F_1}\cdots(\frac{\partial}{F_k\partial x_k})^{a_k}\frac{1}{F_k}\cdots(\frac{\partial}{F_n\partial x_n})^{a_n}\frac{h}{F_n}\bigg)\\
  =\sum_{k=1}^{n}(\frac{\partial}{F_1\partial x_1})^{a_1}\frac{1}{F_1}\cdots\eta_i(\frac{\partial}{F_k\partial x_k})^{a_k+1}\frac{1}{F_k}\cdots(\frac{\partial}{F_n\partial x_n})^{a_n}\frac{h}{F_n}\\
  -\sum_{k=1}^{n}(\frac{\partial}{F_1\partial x_1})^{a_1}\frac{1}{F_1}\cdots(\frac{\partial}{F_k\partial x_k})^{a_k+1}\eta_i\frac{1}{F_k}\cdots(\frac{\partial}{F_n\partial x_n})^{a_n}\frac{h}{F_n}.
\end{multline}
Finally, \eqref{app52} and \eqref{important3} show that for $i=1,\dots,\mu$ and $h,g\in\C[x_1,\dots,x_n]$,
  \begin{equation}\label{im514}
      \frac{\partial}{\partial z_i}\big(b_{k,F}(h)\big)=-b_{k+1,F}(\eta_ih)+\eta_i b_{k+1,F}(h).
      \end{equation}
Replacing $F$ with $-F$, a similar observation shows that
\begin{equation}\label{im555}
 \frac{\partial}{\partial z_i}\big(b_{k,-F}(g)\big)=b_{k+1,-F}(\eta_ig)-\eta_i b_{k+1,-F}(g).
\end{equation}
  From \eqref{imfor57}, we see that \eqref{im514} and \eqref{im555}  imply \eqref{lem53512}.
  \end{proof}

From Lemma~\ref{imlem55} and Lemma~\ref{lem57}, we establish the followings.

\begin{theorem}\label{cojecmain}
$u^n\sH_f$ is Saito's higher residue pairings, i.e.,
$$u^n\sH_f=\sK_f.$$ 
  \end{theorem}

\begin{remark}
One can compare the explicit ones in  \eqref{impo2} from \eqref{impotant}  with a couple of the explicit formulae of Saito's higher residue  pairings found in \cite[Section 2]{ksaito}. We note that there is an investigation about a similar result in terms of polyvector fields instead of differential forms in \cite{Li}.
\end{remark}

\subsection{The conjecture of Shklyarov}\label{sec422}

A work \cite{shkl} by D. Shklyarov sheds light on a relation between the canonical pairing in the dg category of matrix factorizations of an isolated singularity and Saito’s higher residue pairings on the twisted de Rham cohomology associated with the singularity. According to D. Shklyarov \cite{Shk: HRR,shkl}, the periodic cyclic homology $HP(\mathcal{A})$ of any proper $\C$-linear dg category $\mathcal{A}$ has a canonical $\C\uu$-sesquilinear pairing
 $$\sK_\cA:HP_\ast(\cA)\otimes HP_\ast(\cA)\to \C\uu.$$
We remark that $\langle-,-\rangle_{HP(\cA)\otimes HP(\cA^\op)}$ in Section~\ref{defcan} is $\sK_\cA$ after composing an isomorphism $HP(\cA)\to HP(\cA^\op)$; see \cite[Section 2.2]{shkl} for the isomorphism.

Let $Q:=\mathbb{C}[x_1,\dots,x_n]$ and assume that $\fm := (x_1, \dots, x_n)$ is the only critical point of the map $f: \mathbb{A}^n_\C \to \mathbb{A}^1_\C$. The  differential $\Z/2$-graded category $\MF(Q,f)$ of matrix factorizations is proper and there is a natural embedding
$$HN_\ast(\MF(Q,f))\hookrightarrow HP_\ast(\MF(Q,f));$$
see \cite{shkl, HK1} for the definitions of matrix factorizations, the negative cyclic homology, and the periodic cyclic homology. It is known that there are isomorphisms
$$I_f:HN_\ast(\MF(Q,f))\to\cH_{f}^{(0)}\text{ and }I_{-f}:HN_\ast(\MF(Q,-f))\to\cH_{-f}^{(0)};$$
see \cite{shkl} for details. Denoting the restriction of the canonical pairing to the negative cyclic homology $HN_\ast(\MF(Q,f))$ by $\sK_\MF$, in \cite{shkl} D. Shklyarov proves that
$$c_f\cdot u^n\cdot\sK_\MF\big(I_{f}^{-1}(-),I_{f}^{-1}(-)\big)=\sK_f(-,-)$$
and conjectures that $c_f=(-1)^{\frac{n(n+1)}{2}}$. The first proof of the conjecture appears from the work of M. Brown and M. Walker in \cite{BW3}. With a different method, for a global matrix factorization the conjecture is proved by B. Kim; see \cite[Theorem 1.2]{Kim} and also see \cite{HK1} about the investigation about the negative cyclic case and the periodic cyclic case. With different techniques from \cite{BW3,Kim}, we reprove the conjecture of Shklyarov.

\renewcommand\qedsymbol{\textbf{Q.E.D.}}

\begin{theorem}[The conjectrue of Shklyarov]\label{saitomain}
$$(-1)^{\frac{n(n+1)}{2}}\cdot u^n\cdot \sK_\MF\big( I_{f}^{-1}(-),I_{f}^{-1}(-)\big)=\sK_f(-,-).$$
  \end{theorem}

\begin{proof}
  From the uniqueness of the canonical pairing, we obtain that
  $$\langle -,-\rangle_{\ccan}=\sK_\MF\big( I_{f}^{-1}(-),I_{f}^{-1}(-)\big).$$
  By Theorem~\ref{main} and Theorem~\ref{cojecmain}, the proof of the conjecture is established.
\end{proof}

\section*{Acknowledgments}
We thank the anonymous referee for careful reading of the manuscript and the helpful comments and suggestions.

\end{document}